\documentclass[12pt]{amsart}
\usepackage[latin1]{inputenc}
\usepackage{amssymb, amsmath, amscd, mathrsfs,marvosym, psfrag} 

\usepackage{pstricks}

\input xy
\xyoption{all}
\DeclareMathAlphabet{\mathpzc}{OT1}{pzc}{m}{it}

\newtheorem{theorem}{Theorem}[section]

\newtheorem{proposition}[theorem]{Proposition}

\newtheorem{lemma}[theorem]{Lemma}

\theoremstyle{definition}
\newtheorem{definition}[theorem]{Definition}

\theoremstyle{remark}

\def\varle{\leqslant}

\newcommand{\CA}{{\mathcal A}}

\newcommand{\CC}{{\mathcal C}}

\newcommand{\CI}{{\mathcal I}}
\newcommand{\CJ}{{\mathcal J}}

\newcommand{\CO}{{\mathcal O}}

\newcommand{\CW}{{\mathcal W}}

\newcommand{\fh}{{{\mathfrak h}}} 
 
\newcommand{\fp}{{{\mathfrak p}}} 
\newcommand{\fg}{{{\mathfrak g}}} 
\newcommand{\fb}{{{\mathfrak b}}}

\newcommand{\fn}{{{\mathfrak n}}} 
\newcommand{\fm}{{{\mathfrak m}}}

\newcommand{\fhd}{\fh^\star}

\newcommand{\hCW}{{\widehat\CW}}

\newcommand{\hCO}{{\widehat\CO}}

\newcommand{\hfh}{{\widehat\fh}}
\newcommand{\hfg}{{\widehat\fg}}
\newcommand{\hfb}{{\widehat\fb}}
\newcommand{\hfn}{{\widehat\fn}}

\newcommand{\hS}{{\widehat S}}

\newcommand{\hR}{{\widehat R}}
\newcommand{\hPi}{{\widehat \Pi}}

\newcommand{\hfhd}{\widehat{\fh}^\star}

\newcommand{\tS}{{\widetilde{S}}}
\newcommand{\tQ}{{\widetilde{Q}}}

\newcommand{\DC}{{\mathbb C}}

\newcommand{\DZ}{{\mathbb Z}}
\newcommand{\DK}{{\mathbb K}}

\newcommand{\DN}{{\mathbb N}}

\newcommand{\Hom}{{\operatorname{Hom}}}

\newcommand{\Res}{{\operatorname{Res}}}

\newcommand{\catmod}{{\operatorname{-mod}}}

\DeclareMathOperator{\cha}{\mathrm{ch}}

\newcommand{\im}{{\operatorname{im}}}

\newcommand{\ol}{\overline}
\newcommand{\ul}{\underline}

\newcommand{\id}{{\operatorname{id}}}

\newcommand{\re}{{\operatorname{re}}}
\newcommand{\crit}{{\operatorname{crit}}}

\newcommand{\GL}{{\operatorname{GL}}}

\newcommand{\rCO}{{\ol\CO}}
\newcommand{\rP}{{\ol P}}
\newcommand{\rnabla}{{\ol\nabla}}
\newcommand{\rDelta}{{\ol\Delta}}
\newcommand{\tCO}{{\ol{\CO}}}

\newcommand{\comment}[1]{}

\begin{document}

\pagenumbering{arabic}
\title[Critical level representations]{On the restricted projective objects in the affine category $\CO$ at the critical level}
\author[]{Peter Fiebig$^\ast$}
\thanks{$^\ast$partially supported by a grant of the Landesstiftung Baden--W\"urttemberg and by the DFG-Schwerpunkt 1388}

\begin{abstract}  This article gives both an overview and  supplements the articles \cite{AF08} and \cite{AF09} on the critical level category $\CO$ over an affine Kac--Moody algebra. In particular, we study the restricted projective objects and review the restricted reciprocity and linkage principles.
\end{abstract}
\address{Department Mathematik, Universit{ä}t Erlangen-N\"urnberg, 91054 Erlangen, Germany}
\email{fiebig@mi.uni-erlangen.de}
\maketitle

\section{Introduction}
The main objective of this paper is to introduce to and to complement the results of the papers \cite{AF08} and \cite{AF09} on the critical level representation theory of affine Kac--Moody algebras that provide the first steps in a research project, joint with Tomoyuki Arakawa, whose main motivation is the determination of the critical level simple highest weight characters. 

There are at least two (essentially different) approaches to character problems in Lie theory. The first (and slightly more classical), is due to Beilinson and Bernstein and utilizes a {\em localization functor}, i.e.~a functor that realizes representations of a Lie algebra as $D$-modules on a convenient algebraic variety. This functor was initially introduced in \cite{BB81}  in order to determine the characters of simple highest weight representations of semisimple Lie algebras. Later, Kashiwara and Tanisaki used a similar functor in the case of symmetrizable Kac--Moody algebras (cf.~\cite{KT}). For modular Lie algebras, i.e.~Lie algebras over a field of positive characteristic, a version of the localization functor is one of the main ingredients in the work of Bezrukavnikov et al.~(cf.~\cite{BMR08}). Recently, Frenkel and Gaitsgory used it in their formulation of the local geometric Langlands conjectures and their study of the critical level representation theory of affine Kac--Moody algebras (cf.~\cite{FG06}). 

The second approach goes back to Soergel (cf.~\cite{Soe90}). Here, the main idea is to link the representation theory to the topology of an algebraic variety (most notably to the category of perverse sheaves) by an intermediate ``combinatorial'' category. These combinatorial categories often have a slightly artificial flavour, examples of which are categories of Soergel bimodules, of sheaves on moment graphs and the highly complicated category studied in \cite{AJS94}. However, it turned out that they can also play a significant role outside their originial habitat, for example in knot theory or for $p$-smoothness questions in complex algebraic geometry (cf.~\cite{FW10}).

The first example of a relation of this second type  appears in \cite{Soe90} in the case of semisimple complex Lie algebras. The paper \cite{Fie06}  contains the generalization to the symmetrizable, non-critical Kac--Moody case. In the same spirit, the paper \cite{Fie07} links restricted representations of a modular Lie algebra as well as representations of the small quantum group to {\em parity sheaves} on affine flag varieties (one of the main ingredients for this is the combinatorial description in \cite{AJS94}). In our research project we hope to establish a similar result for the critical level representation theory of an affine Kac--Moody algebra. 

Note that both approaches outlined above potentially are sufficient to determine representation theoretic data such as characters, but it is only after one takes them together that they release their full potential. For example, the celebrated Koszul duality for the category $\CO$ of a semisimple Lie algebra is constructed by combining  the Beilinson--Bernstein localization as well as the Soergel approach (cf.~\cite{BGS96}). 

In the following we review our approach and state the main results of the articles \cite{AF08,AF09}. Moreover, we complement these articles by a new and simplified treatment of projective covers in deformed versions of the affine category $\CO$.

\section{Affine Kac--Moody algebras}

We fix a finite dimensional simple complex Lie algebra  $\fg$ and we denote by $k\colon \fg\times\fg\to\DC$ its Killing form. First, we  explain the main steps of the construction of the affine Kac--Moody algebra associated with $\fg$ (for more details, see \cite{Kac}).

The {\em loop algebra}  associated with $\fg$ is the Lie algebra with underlying vector space $\fg\otimes_\DC\DC[t,t^{-1}]$ that is endowed with the $\DC[t,t^{-1}]$-bilinear extension of the bracket of $\fg$. So we have
$$
[x\otimes t^n,y\otimes t^m]=[x,y]\otimes t^{m+n}
$$
for $x,y\in\fg$ and $m,n\in\DZ$. The loop algebra has an up to isomorphism unique non-split central extension $\widetilde \fg$ of rank one. Its underlying vector space is $\fg\otimes_\DC\DC[t,t^{-1}]\oplus\DC K$ and the  bracket is given by
\begin{align*}
[x\otimes t^n,y\otimes t^m]&=[x,y]\otimes t^{m+n}+n\delta_{m,-n} k(x,y) K,\\
[K,\widetilde \fg] &=\{0\}
\end{align*}
(here $ \delta_{a,b}$ denote the Kronecker symbol). 

In order to obtain the affine Kac--Moody algebra $\hfg$ associated with $\fg$ we add the outer derivation operator $D=t\frac{\partial}{\partial t}$ to $\widetilde \fg$. So we obtain  the vector space $\fg\otimes_\DC\DC[t,t^{-1}]\oplus\DC K\oplus \DC D$ with bracket
\begin{align*}
[x\otimes t^n,y\otimes t^m]&=[x,y]\otimes t^{m+n}+n\delta_{m,-n} k(x,y) K,\\
[K, \hfg] &=\{0\},\\
[D,x\otimes t^n]&=nx\otimes t^n
\end{align*}
for $x,y\in\fg$, $m,n\in\DZ$. Note that $\fg$ naturally appears as a subalgebra in $\hfg$ via the map $x\mapsto x\otimes 1$.

\subsection{Roots and coroots}
We fix a Cartan subalgebra $\fh$ in $\fg$ and a Borel subalgebra $\fb\subset\fg$ that contains $\fh$. Then the corresponding Cartan and Borel subalgebras in $\hfg$ are
\begin{align*}
\hfh&:=\fh\oplus\DC K\oplus\DC D,\\
\hfb&:=\fg\otimes t\DC[t]\oplus\fb\oplus\DC K\oplus\DC D.
\end{align*}

Let us denote by $R\subset \fhd=\Hom_\DC(\fh,\DC)$ the root system of $\fg$ with respect to $\fh$ and by $R^+\subset R$ the subset of positive roots, i.e.~the set of roots of $\fb$. Let $\fg=\fh\oplus\bigoplus_{\alpha\in R}\fg_\alpha$ be the root space decomposition. The coroot associated with $\alpha\in R$ is the unique element $\alpha^\vee\in[\fg_\alpha,\fg_{-\alpha}]$ with the property $\langle \alpha,\alpha^\vee\rangle=2$. 

Note that the dual of the projection $\hfh\to\fh$ along the decomposition $\hfh=\fh\oplus\DC K\oplus \DC D$ allows us to view $\fh$ as a subset of $\hfhd$. 
We denote by $\delta\in\hfhd$ the unique element with 
\begin{align*}
\delta(\fh\oplus\DC K)&= \{0\},\\
\delta(D)&= 1.
\end{align*}
Then the set of roots of $\hfg$ with respect to $\hfh$ is 
$\hR=\hR^{\re}\cup\hR^{\im}$, where
\begin{align*}
\hR^{\re}&=\{\alpha+n\delta\mid \alpha\in R, n\in\DZ\},\\
\hR^{\im}&=\{n\delta\mid n\in\DZ,n\ne 0\}.
\end{align*}
The first set is called the set of {\em real roots}, the second set is called the set of {\em imaginary roots}. 
The corresponding root spaces are
\begin{align*}
\hfg_{\alpha+n\delta}=\fg_\alpha\otimes t^n,\\
\hfg_{n\delta}=\fh\otimes t^n.
\end{align*}
Let $\Pi\subset R^+$ be the  set of simple roots. The set of {\em simple affine roots} is then $\hPi= \Pi\cup\{-\gamma+\delta\}$, where $\gamma \in R^+$ is the highest root. The set of {\em positive affine roots} (i.e., the set of roots of $\hfb$) is $\hR^+=R^+\cup\{\alpha+n\delta\mid  \alpha\in R, n>0\}\cup\{n\delta\mid n>0\}$.

For a real root $\alpha+n\delta$ the space $[\hfg_{\alpha+n\delta},\hfg_{-(\alpha+n\delta)}]$ is a one-dimensional subspace of $\hfh$. The coroot associated to $\alpha+n\delta$ is the unique element $(\alpha+n\delta)^\vee\in[\hfg_{-(\alpha+n\delta)}, \hfg_{\alpha+n\delta}]$ with the property $\langle\alpha+n\delta,(\alpha+n\delta)^\vee\rangle=2$. Explicitely, this is 
$$
(\alpha+n\delta)^\vee=\alpha^\vee+\frac{2n}{k(\alpha,\alpha)}K.
$$
Here we denote by $k\colon \fhd\times\fhd\to\DC$ the bilinear form induced by the Killing form.

\subsection{The affine Weyl group}
To a real affine root $\alpha+n\delta$ we associate the following reflection on $\hfhd$:
$$
s_{\alpha+n\delta}(\lambda)=\lambda-\langle\lambda, (\alpha+n\delta)^\vee\rangle(\alpha+n\delta).
$$
The {\em affine Weyl group} is the subgroup $\hCW\subset\GL(\hfhd)$ that is generated by all reflections $s_{\alpha+n\delta}$ with $\alpha+n\delta\in \hR^{\re}$.

We need the following shifted, non-linear action of $\hCW$ on $\hfhd$. Let us choose an element $\rho\in\hfhd$ with the property that $\rho(\alpha^\vee)=1$ for each simple affine coroot $\alpha^\vee$. Note that $\rho$ is not uniquely defined, as the simple coroots do not generate $\hfh$. Instead, $\rho+x\delta$ would do equally well for each $x\in\DC$. However, as $\delta$ is stabilized by the action of $\hCW$, the {\em dot-action}
$$
w.\lambda=w(\lambda+\rho)-\rho
$$
is independent of the choice. It fixes the line $-\rho+\DC\delta$.

\subsection{Simple highest weight characters}

Let $M$ be a $\hfg$-module. For any linear form $\lambda\in\hfhd$ we denote by 
$$
M_\lambda=\{m\in M\mid H.m=\lambda(H)m\text{ for all $H\in\hfh$}\}
$$ 
the eigenspace of the $\hfh$-action on $M$ with eigenvalue $\lambda$. 

The set $\hfhd$ carries a natural partial order (with respect to our choice of $\fb$): we set $\lambda\ge \mu$ if and only if $\lambda-\mu$ can be written as a sum of positive roots, i.e.~if and only if $\lambda-\mu\in\DZ_{\ge 0}\hR^+$.

It is well-known that there is an up to isomorphism unique simple $\hfg$-module $L(\lambda)$ {\em with highest weight $\lambda$}, i.e.~that is generated by its $\lambda$-weight space and has the property that $L(\lambda)_\mu\ne 0$ implies $\lambda\ge \mu$. Then the complex dimension of $L(\lambda)_\mu$ is finite, hence we can consider the formal sum
$$
\cha L(\lambda)=\sum_{\mu\le \lambda} \dim_\DC L(\lambda)_\mu e^{\mu}
$$
as an element in the formal completion (with respect to $\varle$) of the group algebra of the additive group $\hfhd$. 

Now, if $\lambda$ is integral and dominant (i.e., if $\langle\lambda,\alpha^\vee\rangle\in\DZ_{\ge 0}$ for all simple affine roots $\alpha$), then $\cha L(\lambda)$ is given by the Weyl-Kac character formula (cf.~\cite{Kac}). More generally, if $\lambda$ is non-critical (i.e., if $\langle\lambda+\rho,K\rangle\ne 0$), then $\cha L(\lambda)$ is given by an appropriate version of the formula conjectured by Kazhdan and Lusztig (cf.~\cite{KT}). In the case that $\lambda$ is critical, Feigin and Frenkel conjectured a formula for $\cha L(\lambda)$ (see \cite{AF08}). This conjecture, however, is yet unproven in general. In the integral dominant critical case, the conjecture follows from the results in  \cite{FG09}. The objective of this paper is to supplement the papers \cite{AF08,AF09} that provide a first step in a program that aims to solve the character problem at the critical level.

\subsection{The  category $\hCO$}

It is most convenient to introduce now a categorical framework for the above mentioned  problem.

\begin{definition}
\begin{enumerate} 
\item $M$ is called a {\em weight module} (with respect to $\hfh$), if $\hfh$ acts semisimply, i.e.~if $M=\bigoplus_{\lambda\in\hfhd}M_\lambda$. 
\item $M$ is called {\em locally $\hfb$-finite}, if each element of $M$ is contained in a finite dimensional $\hfb$-submodule.  
\end{enumerate}
\end{definition}

We denote by $\hCO$ the full subcategory of all $\hfg$-modules  which are weight modules and on which  $\hfb$ acts locally finitely. This is an abelian subcategory of $\hfg\catmod$. Each simple object $L(\lambda)$ is contained in $\hCO$, as is, more generally, each module with highest weight. For $\lambda\in\hfhd$ the {\em Verma module} with highest weight $\lambda$ is defined as
$$
\Delta(\lambda)=U(\hfg)\otimes_{U(\hfb)} \DC_\lambda,
$$
where $\DC_\lambda$ is the simple $\hfb$-module corresponding to the character $\hfb\to\hfh\stackrel{\lambda}\to\DC$, where the map on the left is the  homomorphism of Lie algebras that is left invers to the inclusion $\hfh\subset\hfb$. 
The {\em dual Verma module} $\nabla(\lambda)$ is the restricted dual of $\Delta(\lambda)$, i.e.~it is the set of $\hfh$-finite vectors in the representation of $\hfg$ that is dual to $\Delta(\lambda)$. Each $\nabla(\lambda)$ is contained in $\hCO$ as well. 

\subsection{The level}
For a $\hfg$-module $M$ and a complex number $k$ we define 
$$
M_k:=\{m\in M\mid K.m=km\},
$$ 
the eigenspace of the action of the central element $K$ in $\hfg$ with eigenvalue $k$. Clearly, each  $M_k$ is a submodule in $M$. A $\hfg$-module $M$ is said to be {\em of level $k$} if $M=M_k$.

 If $M$ is a weight module, then $K$ acts semisimply on $M$, so $M=\bigoplus_{k\in\DC} M_k$. In fact, in this case we have $M_k=\bigoplus_{\lambda\in\hfhd,\lambda(K)=k} M_\lambda$. If we denote by $\hCO_k$ the full subcategory of $\hCO$ that contains all modules of level $k$, then the functor
\begin{align*}
\prod_{k\in\DC}\hCO_k&\to\hCO\\
\{M_k\}_{k\in\DC}&\mapsto\bigoplus_{k\in\DC}M_k
\end{align*}
is an equivalence of categories.

\subsection{A graded structure}
In the following we construct a {\em grading functor} $T$ on $\hCO$ (i.e., an autoequivalence $T\colon \hCO \to\hCO$). 
Let us consider the simple $\hfg$-module $L(\delta)$ with highest weight $\delta$. It is one-dimensional. In fact, the algebra $\widetilde\fg=[\hfg,\hfg]=\fg\otimes\DC[t,t^{-1}]\oplus\DC K$ acts trivially on $L(\delta)$, while $D\in\hfg$ acts as the identity operator. 
Recall the usual tensor structure on the category of $\hfg$-modules: If $M$ and $N$ are $\hfg$-modules, then $M\otimes_\DC N$ becomes a $\hfg$-module with the action determined by $X.(m\otimes n)=(X.m)\otimes n+m\otimes X.n$ for $X\in\hfg$ and $m\in M$, $n\in N$. 

We define the functor
\begin{align*}
T\colon \hfg\catmod&\to\hfg\catmod,\\
M&\mapsto M\otimes_\DC L(\delta)
\end{align*}
with the obvious action on morphisms. It is an equivalence with inverse $T^{-1}\colon M\mapsto M\otimes_\DC L(-\delta)$, and it preserves weight modules, as $(TM)_{\lambda}=M_{\lambda-\delta}\otimes_\DC L(\delta)$ for each $\hfg$-module $M$ and $\lambda\in\hfhd$. Moreover, if $N\subset M$ is a $\hfb$-submodule, then $N\otimes_\DC L(\delta)\subset M\otimes_\DC L(\delta)$ is a $\hfb$-submodule. Hence $T$ also preserves the $\hfb$-local finiteness condition. So  $T$ preserves the category $\hCO$,  hence it makes $\hCO$ into a {\em graded category}. As $\langle \delta,K\rangle=0$, the functor $T$ in addition preserves the level, i.e.~it induces a grading on the subcategories $\hCO_k$ for all $k\in\DC$.

\subsection{The graded center}
With any graded   category $(\CC,T)$ we can associate the following:

\begin{definition} The graded center of $(\CC,T)$  is the graded vector space $\CA=\CA(\CC,T)=\bigoplus_{n\in\DZ} \CA_n$,
where $\CA_n$ is the space of natural transformations $\tau\colon \id_{\CC}\to T^n$ with the property that for all objects $M$ of $\CC$ and all $m\in\DZ$, the morphism
$T^m(\tau^M)\colon T^m M\to T^mT^n M=T^{n}T^m M$ equals the morphism $\tau^{T^m M}$. 
\end{definition}
Note that $\CA$ carries a natural multiplication that makes it into a commutative, associative, unital algebra (cf.~\cite{AF08}). 

In particular, we have the graded centers $\CA$ of $(\hCO,T)$ and  $\CA_k$ of $(\hCO_k,T)$ for all $k\in\DC$. Clearly, $\CA=\prod_{k\in\DC} \CA_k$. Now there is only one value $k=\crit:=\langle-\rho,K\rangle$ for which $(\CA_k)_n$ is non-zero for all $n\in\DZ$. In fact, if $k\ne\crit$, then $(\CA_k)_n$ is the trivial vector space for all $n\ne 0$ (for more information about $(\CA_k)_0$ for $k\ne\crit$, see \cite{FieMZ}). However,  the spaces $(\CA_{\crit})_n$ are huge for any $n\ne 0$ (cf.~\cite{AF08}).

\subsection{The restricted category $\rCO$} In the following we abbreviate $\CA=\CA(\hCO,T)$. We now define the restricted subcategory $\rCO$ of $\hCO$ as a special fiber for the action of $\CA$.
\begin{definition} An object $M$ of $\hCO$ is called {\em restricted} if for all $n\ne 0$ and $\tau\in \CA_n$ we have that $\tau^M\colon M\to T^n M$ is the zero homomorphism.
\end{definition}

We denote by $\rCO\subset\hCO$ the full subcategory that contains all restricted objects and by $\rCO_k\subset\rCO$ the full subcategory of restricted modules of level $k$.
Note that if $k\ne\crit$, then each $M\in\hCO_k$ is restricted, i.e.~$\rCO_k=\hCO_k$, as $(\CA_k)_n=0$ for $n\ne 0$. A simple highest weight module $L(\mu)$ is always restricted.

The inclusion functor $\rCO\to \hCO$ has adjoints on both sides. For $M\in\hCO$ we denote by $M^\prime$ the submodule of $M$ that is generated by the images of all homomorphisms $T^{-n}\tau^M\colon T^{-n}M\to M$ with $\tau\in\CA_n$ and $n\ne 0$. Set $\ol M:= M/M^\prime$. Dually, denote by $\ul M$ the submodule of $M$ that contains all elements $m$ with $\tau^M(m)=0$ for all $\tau\in\CA_n$, $n\ne 0$. Then these definitions extend to functors $\ol\cdot,\ul\cdot\colon\hCO\to\rCO$.
The next result follows  easily  from the definitions.

\begin{lemma}  The functor $M\mapsto \ol M$ is left adjoint to the inclusion functor $\rCO\subset\hCO$, and the functor $M\mapsto \ul M$ is right adjoint to the inclusion functor.
\end{lemma}

We define the {\em restricted Verma module} corresponding to $\lambda\in\hfhd$ as
$$
\rDelta(\lambda):=\ol{\Delta(\lambda)}
$$
 and the {\em restricted dual Verma module} as
 $$
 \rnabla(\lambda):=\ul{\nabla(\lambda)}.
 $$
 
\section{Deformed category $\CO$}

One of the main methods in our approach to the representation theory of Kac--Moody algebras is the following deformation idea. Let us denote by $S=S(\fh)$ and $\hS:=S(\hfh)$ the symmetric algebras associated with the vector spaces $\fh$ and $\hfh$. The projection $\hfh\to\fh$ along the decomposition $\hfh=\fh\oplus\DC K\oplus \DC D$ yields a homomorphism $\hS\to S$ of algebras. Let $A$ be a commutative, Noetherian, unital $S$-algebra. Then we can consider $A$ as an $\hS$-algebra via the above homomorphism. We call such an algebra in the following a  {\em deformation algebra}. As $A$ contains a unit, we have a canonical map $\tau\colon  \hfh\to A$, $f\mapsto f.1_A$. Note that, as we start with an $S$-algebra, we will always have $\tau(K)=\tau(D)=0$.

By a $\hfg_A$-module we mean in the following an $A$-module together with  an action of $\hfg$ that is $A$-linear, i.e.~a $\hfg$-$A$-bimodule. 
Let $M$ be a $\hfg_A$-module and $\lambda\in\hfhd$. We define the $\lambda$-weight space of $M$ as
$$
M_\lambda:=\{m\in M\mid H.m=(\lambda(H)+\tau(H))m\text{ for all $H\in\hfh$}\}.
$$
(Note that it would be more appropriate to  call this the $\lambda+\tau$-weight space.)

Let $M$ be a $\hfg_A$-module.
\begin{definition} 
\begin{enumerate}
\item
We say that $M$ is a {\em weight module} if $M=\bigoplus_{\lambda\in\hfhd} M_\lambda$.
\item We say that $M$ is {\em locally $\hfb_A$-finite}, if every element in $M$ is contained in a $\hfb_A$-submodule that is finitely generated as an $A$-module.
\end{enumerate}
\end{definition}
We denote by $\hCO_A$ the full subcategory of the category of $\hfg_A$-modules that contains all locally $\hfb_A$-finite weight modules.

Note that if $A=\DK$ is a field, then $\hCO_\DK$ is a direct summand of the usual category $\hCO$ defined for the Lie algebra $\hfg\otimes_\DC \DK$. It contains all modules $M$ whose weights have the special form $\lambda+\tau$ with $\lambda\in\hfhd$ (note that the later element can be considered as a $\DK$-linear form on the Cartan subalgebra $\hfh\otimes_\DC \DK$).

\subsection{Verma modules in $\hCO_A$}
Let $\lambda\in\hfhd$ and denote by $A_\lambda$ the $\hfb_A$-module that is free of rank one as an $A$-module and on which $\hfh$ acts by the character $\lambda+\tau$ and $[\hfb,\hfb]$ acts trivially. The {\em deformed Verma module}   with highest weight $\lambda$  is
$$
\Delta_A(\lambda):=U(\hfg)\otimes_{U(\hfb)} A_\lambda.
$$
Then $\Delta_A(\lambda)$ is a weight module  and each weight space $\Delta_A(\lambda)_\mu$ is a  free $A$-module of finite rank. Moreover, $\Delta_A(\lambda)$ is $\hfb_A$-locally finite, so it appears as an object in $\hCO_A$. If $A\to A^\prime$ is a homomorphism of deformation algebras, then $\Delta_A(\lambda)\otimes_A A^\prime\cong \Delta_{A^\prime}(\lambda)$. 

\begin{definition} We say that an object $M$ of $\hCO_A$ {\em admits a Verma flag} if there is a finite filtration
$$
0=M_0\subset M_1\subset \dots\subset M_n=M
$$
such that for each $i=1,\dots, n$ the quotient $M_i/M_{i-1}$ is isomorphic to a deformed Verma module $\Delta_A(\mu_i)$ for some $\mu_i\in\hfhd$. 
\end{definition}

It turns out that the multiset $\{\mu_1,\dots,\mu_n\}$ is independent of the chosen filtration and hence for each $M$ that admits a Verma flag the multiplicity
$$
(M:\Delta_A(\nu)):=\#\{i\mid \mu_i=\nu\}
$$
is well defined for all $\nu\in\hfhd$.
 
\subsection{Simple objects in $\hCO_A$}
Let us now  assume that $A$ is a {\em local} deformation algebra with maximal ideal $\fm$ and residue field $\DK=A/\fm$. Note that we can consider each $\hfg_\DK$-module as a $\hfg_A$-module on which $A$ acts via the quotient map $A\to\DK$. This even extends to a functor  $\Res\colon \hCO_\DK\to\hCO_A$. It is clearly isomorphic to the unique simple quotient of $\Delta_A(\lambda)$. 

For $\lambda\in\hfhd$ we have the simple quotient $L_\DK(\lambda)$ of the Verma module $\Delta_\DK(\lambda)$ in $\hCO_\DK$ and we define $L_A(\lambda):=\Res(L_\DK(\lambda))$. 
\begin{proposition} [\cite{FieMZ}] The set $\{L_A(\lambda)\}_{\lambda\in\hfhd}$ is a full set of representatives for the simple objects in $\hCO_A$.
\end{proposition} 

\subsection{The restricted deformed category $\rCO_A$}

As in the non-deformed situation we can define the shift functor $T$ on $\hCO_A$ that maps an object $M$ to $M\otimes_\DC L(\delta)$  and has the obvious impact on morphisms. Then $T\colon\hCO_A\to\hCO_A$ is an equivalence and we obtain the graded $A$-algebra $\CA_A=\CA(\hCO_{A},T)$. 
An object $M$ of $\hCO_{A}$ is called {\em restricted} if $\tau^M\colon M\to T^{n}M$ is the zero morphism for each $\tau\in(\CA_A)_{n}$, $n\ne 0$. We define $\tCO_{A}$ as the full subcategory of $\hCO_{A}$ that contains all restricted objects. 

As before there is a functor $\ol\cdot\colon\hCO_A\to\rCO_A$ that is left adjoint to the inclusion functor (it is constructed as in the non-restricted case). The {\em deformed restricted Verma module} with highest weight is 
$$
\rDelta_A(\lambda):=\ol{\Delta_A(\lambda)}.
$$

\section{Projective objects in $\hCO_A$}
Let $A$ be a local deformation algebra. Now we want to study projective objects in the deformed category $\hCO_A$. In particular, we want to study projective covers of simple objects $L_A(\lambda)$, i.e.~projective objects $P$ together with a surjective map $c\colon P\to L_A(\lambda)$ with the following property: If $g\colon M\to P$ is a homomorphism such that $c\circ g\colon M\to L_A(\lambda)$ is surjective, then $g$ is surjective. Such projective covers do not always exists in $\hCO_A$. But when we restrict ourselves to {\em truncated categories}, then the situation improves. 

\subsection{Truncated subcategories}
The truncations that we are going to consider are  associated to open, locally bounded subsets of $\hfhd$. 

\begin{definition} A subset $\CJ$ of $\hfhd$ is called {\em open}, if for all $\lambda\in\CJ$ and all $\mu\in\hfhd$ with $\mu\le\lambda$ we have  $\mu\in\CJ$. A subset $\CJ$ is {\em locally bounded}, if for all $\mu\in\hfhd$ the set $\CJ_{\ge\mu}:=\{\gamma\in\CJ\mid \gamma\ge\mu\}$ is finite.
\end{definition}
 
Let us fix an open, locally bounded subset $\CJ$ of $\hfhd$. We define the full subcategory $\hCO_A^{\CJ}$ of $\hCO_A$ that contains all objects $M$ with the property that $M_\lambda\ne0$ implies $\lambda\in \CJ$. 

Let $M\in\hCO_A$ and let $M^\prime$ be the submodule of $M$ that is generated by the weight spaces $M_\nu$ with $\nu\not\in \CJ$ and set $M^\CJ:=M/M^\prime$. Then $M\mapsto M^\CJ$ is a functor from $\hCO_A$ to $\hCO_A^\CJ$ that is left adjoint to the inclusion functor $\hCO_A^\CJ\subset\hCO_A$.
 
\subsection{Existence of projective covers}
The main objective of this section is to give a new proof of the following result.  
\begin{theorem}\label{theorem-projcov} Suppose that $A$ is a local deformation algebra. Let $\mu\in\CJ$. Then there exists a projective cover $P_A^{\CJ}(\mu)\to L_A(\mu)$ in $\hCO_A^{\CJ}$. 
\end{theorem}

In order to prove the above theorem, we first consider the universal enveloping algebra $U(\hfb)$ under the adjoint action of $\hfh$, so we obtain a weight space decomposition $U(\hfb)=\bigoplus_{\gamma\in\DZ_{\ge 0}\hR^+}U(\hfb)_\gamma$. For $\mu\in\CJ$ we define $\CJ^\prime=\CJ-\mu=\{\nu\in\hfhd\mid \nu=\gamma-\mu\text{ for some $\gamma\in\CJ$}\}$ and $\CI^\prime=\hfhd\setminus\CJ^\prime$. Then the vector space $U(\hfb)_{\CI^\prime}=\bigoplus_{\gamma\in\CI^\prime}U(\hfb)_\gamma$ is a (two-sided) ideal in $U(\hfb)$, and hence $U(\hfb)^{\CJ^\prime}=U(\hfb)/U(\hfb)_{\CI^\prime}$ is a $U(\hfb)$-module. As $\hS=U(\hfh)$, we get a (right) action of $\hS$ on $U(\hfb)^{\CJ^\prime}$, and hence we can form the tensor product 
$$
Q_A^{\CJ}(\mu):=U(\hfg)\otimes_{U(\hfb)} U(\hfb)^{\CJ^\prime}\otimes_{\hS} A_\mu.
$$
This is a $\hfg_A$-module. As in \cite{RCW} (see also \cite{FieMZ}) one shows that this object represents the functor $\hCO_A^\CJ\to A\catmod$, 
 $M\mapsto M_\mu$, so by the definition of $\hCO_A$ it is projective in $\hCO_A$. Moreover, it admits a Verma flag with multiplicities
$$
(Q_A^{\CJ}(\mu):\Delta_A(\nu))=\begin{cases}
\dim_\DC U(\hfn)_{\nu-\mu},&\text{ if $\nu\in\CJ$},\\
0,&\text{ if $\nu\not\in\CJ$},
\end{cases}
$$
where $\hfn=\bigoplus_{\alpha\in\hR^+}\hfg_\alpha$. In particular, we have $(Q_A^{\CJ}(\mu):\Delta_A(\nu))\ne0$ only if $\nu\ge\mu$ and $(Q_A^{\CJ}(\mu):\Delta_A(\mu))=1$.

Every direct summand of a module with a Verma flag admits a Verma flag as well. As $\Delta_A(\mu)$ occurs with multiplicity one, there is a direct summand  $P^{\CJ}_A(\mu)$ of $Q_A^{\CJ}(\mu)$ with 
$$
(P^{\CJ}_A(\mu):\Delta_A(\mu))=1.
$$
Note that we do not yet claim that $P^{\CJ}_A(\mu)$ is unique up to isomorphism, yet this will be a consequence once we proved Theorem \ref{theorem-projcov}. For now, it suffices to choose a direct summand with the above properties.

As all other Verma subquotients of $P^{\CJ}_A(\mu)$ have highest weights $\mu^\prime$ with $\mu^\prime>\mu$, there is a surjection $P^{\CJ}_A(\mu)\to\Delta_A(\mu)$, hence a surjection $P^{\CJ}_A(\mu)\to L_A(\mu)$ and this surjection is unique up to non-zero scalars in $\DC$. We can now prove the above  theorem.

\begin{proof}[Proof of  Theorem \ref{theorem-projcov}] We prove the statement by induction on the number of elements in the set $\CJ_{\ge \mu}$. If it contains only the element $\mu$, then $P_A^{\CJ}(\mu)=Q_A^{\CJ}(\mu)\cong \Delta_A(\mu)$ and the locality of $A$ implies that $\Delta_A(\mu)\to L_A(\mu)$ is a projective cover. 

So let us fix $\mu\in\CJ$ and let us assume that the statement is proven for all pairs $\mu^\prime\in\CJ^\prime$ such that $\CJ^\prime_{\ge \mu^\prime}$ contains strictly less elements then $\CJ_{\ge\mu}$. As a next step we prove that $L_A(\mu)$ is then the only simple quotient of $P_A^\CJ(\mu)$. Suppose that this is not the case, hence that there exists a surjection $P_A^\CJ(\mu)\to L_A(\nu)$ for some $\nu\ne \mu$. As $P_A^{\CJ}(\mu)$ is generated by its weight spaces corresponding to weights in $\CJ_{\ge\mu}$, this implies $\nu\in\CJ_{\ge \mu}$. By induction assumption, $P_A^\CJ(\nu)\to L_A(\nu)$ is a projective cover. Now by the projectivity of $P_A^\CJ(\mu)$ there is a homomorphism $P_A^{\CJ}(\mu)\to P_A^{\CJ}(\nu)$ such that the diagram

\centerline{
\xymatrix{
P_A^{\CJ}(\mu)\ar[rr]\ar[dr]&& P_A^{\CJ}(\nu)\ar[dl]\\
&L_A(\nu)
}
}
\noindent
commutes. As $P_A^{\CJ}(\nu)\to L_A(\nu)$ is a projective cover, the homomorphism $P_A^\CJ(\mu)\to P_A^\CJ(\nu)$ is surjective, and from the projectivity of $P_A^{\CJ}(\nu)$ and the indecomposability of $P_A^\CJ(\mu)$ we deduce $P_A^\CJ(\mu)\cong P_A^\CJ(\nu)$, which contradicts what we already know about the Verma subquotients of both objects. 
Hence we have proven that $L_A(\mu)$ is the only simple quotient of $P_A^\CJ(\mu)$. 

Now let $\CJ^\prime\subset\CJ$ be another open subset of $\hfhd$. Then $P_A^\CJ(\mu)^{\CJ^\prime}$ is a quotient of $P_A^\CJ(\mu)$ and it is projective in $\hCO_A^{\CJ^\prime}$, as the functor $M\mapsto M^{\CJ^\prime}$, $\hCO_A^\CJ\to\hCO_A^{\CJ^\prime}$, is left adjoint to the inclusion functor. As $P_A^\CJ(\mu)$ has a unique simple quotient, we have $P_A^\CJ(\mu)^{\CJ^\prime}\cong P_A^{\CJ^\prime}(\mu)$.

Now we can prove that $c\colon P^\CJ_A(\mu)\to L_A(\mu)$ is a projective cover. So let $g\colon M\to P_A^\CJ(\mu)$ be a homomorphism such that $c\circ g\colon M\to L_A(\mu)$ is surjective. Then the projectivity implies that there is a homomorphism $h\colon P_A^\CJ(\mu)\to M$ such that the diagram

\centerline{
\xymatrix{
P_A^{\CJ}(\mu)\ar[r]^{h}\ar[dr]& M\ar[r]^g&P_A^{\CJ}(\mu)\ar[dl]\\
&L_A(\mu)
}
}
\noindent
is commutative. We will now prove that the composition $f=g\circ h$ is surjective, from which the surjectivity of $g$ readily follows. Let $\nu\in\CJ$ be a maximal element and consider $\CJ^\prime=\CJ\setminus\{\nu\}$. Then $f^{\CJ^\prime}$ is an endomorphism of $P_A^\CJ(\mu)^{\CJ^\prime}\cong P_A^{\CJ^\prime}(\mu)$. By induction we know that $P_A^{\CJ^\prime}(\mu)\to L_A(\mu)$ is  a projective cover (in $\hCO_A^{\CJ^\prime}$), hence $f^{\CJ^\prime}$ is an automorphism. Hence the quotient $P_A^{\CJ}(\mu)/\im f$ is generated by its $\nu$-weight space. But this quotient then has to be trivial, as $P_A^\CJ(\mu)$ has no simple quotient of highest weight $\nu$. Hence $f$ is surjective, which is what we wanted to show. 
\end{proof}

\subsection{Restricted projective covers} Again we suppose that $A$ is a local deformation algebra. 
Each simple object $L_A(\lambda)$ is restricted, and $\{L_A(\lambda)\}_{\lambda\in\hfhd}$ is  a full set of representatives of the simple objects in $\rCO_A$ as well.  We will now show that projective covers also exist in the truncated restricted categories $\rCO_A^\CJ=\rCO_A\cap\hCO_A^\CJ$.

\begin{theorem} Suppose that $A$ is a local deformation algebra and let $\CJ\subset\hfhd$ be an open, locally bounded subset. Then there exists for each $\lambda\in\CJ$ a projective cover $\rP_A^\CJ(\lambda)\to L_A(\lambda)$ in $\rCO_A^\CJ$.
\end{theorem}
\begin{proof} Consider the projective cover $P_A^\CJ(\lambda)\to L_A(\lambda)$ in $\CO_A^\CJ$ and consider its restriction $\ol{P_A^\CJ(\lambda)}\to \ol{L_A(\lambda)}=L_A(\lambda)$. As the functor $M\mapsto \ol M$ is left adjoint to the (exact) inclusion functor $\rCO_A\subset \CO_A$, $\ol{P_A^\CJ(\lambda)}$ is projective in $\rCO_A$. 

We have seen in the proof of Theorem \ref{theorem-projcov}  that $L_A(\lambda)$ is the only simple subquotient of $P_A^{\CJ}(\lambda)$. Hence $\ol{P_A^\CJ(\lambda)}$ is indecomposable. Now we show that $\rP_A^\CJ(\lambda):=\ol{P_A^\CJ(\lambda)}\to L_A(\lambda)$ is a projective cover. As we have seen in the proof of  Theorem \ref{theorem-projcov},  for this it is enough to show that if $f$ is an endomorphism of $\rP_A^\CJ(\lambda)$ that has the property that the composition $\rP_A^\CJ(\lambda)\stackrel{f}\to\rP_A^\CJ(\lambda)\to L_A(\lambda)$ is surjective, then $f$ is surjective. By projectivity of $P_A^\CJ(\lambda)$ we can find an endomorphism $f^\prime$ of $P_A^\CJ(\lambda)$ such that the diagram

\centerline{
\xymatrix{
P_A^{\CJ}(\lambda)\ar[r]^{f^\prime}\ar[d]& P_A^{\CJ}(\lambda)\ar[d]\\
\rP_A^\CJ(\lambda)\ar[r]^f&\rP_A^\CJ(\lambda)
}
}
\noindent
is commutative (here the vertical maps are the quotient maps). But now $f^\prime$ is surjective, as we have seen in the proof of Theorem \ref{theorem-projcov}, so $f$ has to be surjective as well.
\end{proof}

\section{Block decomposition of $\hCO_A$}

We denote by $\sim_A$ the equivalence relation on $\hfhd$ that is generated by  $\lambda\sim_A\mu$ if there exists some open, locally bounded subset $\CJ$ of $\hfhd$ and a non-zero homomorphism $P_A^{\CJ}(\lambda)\to P_A^{\CJ}(\mu)$. As $P_A^{\CJ}(\lambda)\to L_A(\lambda)$ is a projective cover in $\hCO_A^{\CJ}$, this condition is equivalent to the fact that $L_A(\lambda)$ occurs as a subquotient of $P_A^{\CJ}(\mu)$ (in the following we write $[P_A^{\CJ}(\mu):L_A(\lambda)]\ne 0$ if this is the case).  

For an equivalence class $\Lambda\subset\hfhd$ with respect to $\sim_A$ we define the full subcategory $\hCO_{A,\Lambda}$ of $\hCO_A$ that contains all objects $M$ with the property that if $L_A(\lambda)$ is a subquotient of $M$, then  $\lambda\in \Lambda$.

\begin{theorem}[Block decomposition] The functor
\begin{align*}
\prod_{\Lambda\in\hfhd/_{{\sim_A}}} \hCO_{A,\Lambda}&\to \hCO_A,\\
\{M_\Lambda\}&\mapsto\bigoplus M_\Lambda
\end{align*}
is an equivalence of categories. 
\end{theorem}
\begin{proof} For an equivalence class $\Lambda$ and an object $M$ of $\hCO_A$ let $M_\Lambda$ be the submodule of $M$ that is generated by the images of all homomorphisms $P_A^{\CJ}(\lambda)\to M$ with $\lambda\in\Lambda$ and arbitrary open, locally bounded $\CJ$. By definition of $\sim_A$ the sum of all  $M_\Lambda$ is direct. Moreover, we have  $M=\bigoplus_\Lambda M_\Lambda$, as $M$  is isomorphic to a quotient of a direct sum of various  $P_A^{\CJ}(\lambda)$'s.
\end{proof}

\subsection{Restricted block decomposition}

The block decomposition above  has an immediate analogue in the restricted case: We define the equivalence relation $\ol\sim_A$ on $\hfhd$ as generated by $\lambda\ol\sim_A \mu$ if $[\rP_A^\CJ(\lambda):L_A(\mu)]\ne 0$ for some open, locally bounded subset $\CJ$. As before one proves the following result.
\begin{theorem}[Restricted block decomposition] The functor
\begin{align*}
\prod_{\Lambda\in\hfhd/_{{\ol\sim_A}}} \rCO_{A,\Lambda}&\to \rCO_A,\\
\{M_\Lambda\}&\mapsto\bigoplus M_\Lambda
\end{align*}
is an equivalence of categories. 
\end{theorem}

\subsection{BGGH-reciprocity}
The {\em linkage principle} and the {\em restricted linkage principle} now describe the equivalence classes under $\sim_A$ and $\ol\sim_A$ explicitely. The first  step towards these results are the respective BGGH-reciprocity statements.

\begin{theorem}[Deformed BGGH-reciprocity] 
Let $A$ be a local deformation algebra with residue field $\DK$ and let $\CJ$ be an open locally bounded subset of $\hfhd$ and $\mu\in\CJ$. Then  $P_A^{\CJ}(\mu)$ admits a Verma flag and we have 
$$
(P_A^{\CJ}(\mu):\Delta_A(\lambda))=
\begin{cases}
[\nabla_\DK(\lambda):L_\DK(\mu)],&\text{ if $\lambda\in\CJ$},\\
0,&\text{ if $\lambda\not\in\CJ$}
\end{cases}
$$
for all $\lambda\in\hfhd$. 
\end{theorem} 
Note that the right hand side refers to the $\DK$-linear versions of the objects. 

\begin{proof} By construction, $P_A^\CJ(\mu)$ is a direct summand of an object that admits a Verma flag, hence also admits a Verma flag. In the proof of Theorem \ref{theorem-projcov}  we have shown that $P_A^{\CJ}(\mu)$ has a unique simple quotient, hence $P_A^{\CJ}(\mu)\otimes_A \DK$ must be indecomposable. As it is a direct summand of the projective object $Q_\DK^\CJ(\mu)$ in $\hCO_{\DK}^\CJ$, it  is projective. Hence it has to be isomorphic to $P_\DK^{\CJ}(\mu)$. Hence the Verma multiplicities of $P_A^\CJ(\mu)$ and $P_\DK^\CJ(\mu)$ coincide. So it suffices to prove the above statement in the case that $A=\DK$ is a field, in which case it reduces to the well-known non-deformed BGGH-reciprocity.
\end{proof}

The following is the restricted analogue of the above theorem. 
\begin{theorem}[Restricted BGGH-reciprocity] Let $A$ be a local deformation algebra and $\CJ$  an open locally bounded subset of $\hfhd$ and $\mu\in\CJ$. Then $\rP_A^{\CJ}(\mu)$ admits a restricted Verma flag and for the multiplicities we have
$$
(\rP_A^{\CJ}(\mu):\rDelta_A(\lambda))=
\begin{cases}
[\rnabla_\DK(\lambda):L_\DK(\mu)],&\text{ if $\lambda\in\CJ$},\\
0,&\text{ otherwise}
\end{cases}
$$
for all $\lambda\in\hfhd$. 
\end{theorem} 
The proof can be found in \cite{AF09}.

\subsection{The equivalence relation}

Let us define the equivalence relation $\sim_A^\prime$ on $\hfhd$ as generated by $\lambda\sim_A^\prime\mu$ if $[\Delta_\DK(\lambda):L_\DK(\mu)]\ne 0$. 
\begin{lemma} We have $\sim_A^\prime=\sim_A$. 
\end{lemma}
\begin{proof}
Since the characters of $\Delta_\DK(\lambda)$ and $\nabla_\DK(\lambda)$ coincide, the BGGH-reciprocity implies that we have $\sim_A^\prime=\sim_A^{\prime\prime}$, where $\sim_A^{\prime\prime}$ is generated by $\lambda\sim_A^{\prime\prime}\mu$ if there exists some $\CJ$ with $(P_A^{\CJ}(\lambda):\Delta_A(\mu))\ne 0$. Hence we have to show that $\sim_A=\sim_A^{\prime\prime}$. It is clear that $\lambda\sim_A^{\prime\prime}\mu$ implies $\lambda\sim_A\mu$. So let us suppose that $[P_A^{\CJ}(\lambda):L_A(\mu)]\ne 0$. Then there is a non-zero homomorphism $P_A^{\CJ}(\mu)\to P_A^{\CJ}(\lambda)$. So there must be a Verma subquotient of $P_A^{\CJ}(\lambda)$ that admits a non-zero homomorphism from $P_A^{\CJ}(\mu)$, so if its highest weight is $\nu$, then $\nu\sim_A^{\prime\prime}\lambda$ and there is a non-zero homomorphism $P_A^{\CJ}(\mu)\to\Delta_A(\nu)$. This implies that $[\Delta_\DK(\nu):L_\DK(\mu)]=[\nabla_\DK(\nu):L_\DK(\mu)]\ne 0$, hence $\mu\sim_A^{\prime\prime}\nu$, again by BGGH-reciprocity. So $\lambda\sim_A^{\prime\prime}\mu$. 
\end{proof}

Moreover, the restricted version of the above statement holds as well: using analogous arguments (in particular, using the restricted BGGH-reciprocity) one can prove that $\ol\sim_A$ is also generated by $\lambda\ol\sim_A \mu$ if $[\rDelta_\DK(\lambda):L_\DK(\mu)]\ne 0$. Note, however, that $\ol\sim_A$ is a finer relation than $\sim_A$, i.e.~$\lambda\ol\sim_A\mu$ implies $\lambda\sim_A\mu$.

\section{The linkage principle}

In some sense, the results of the previous section are quite abstract and do not give us enough information about the structure of category $\hCO$. The  next step is to prove the {\em linkage principles}, i.e.~to determine the equivalence classes with respect to $\sim_A$ and $\ol\sim_A$. In the non-restricted case, the linkage principle follows from our results above together with a theorem of Kac and Kazhdan.

\subsection{The theorem of Kac and Kazhdan}
Let $A$ be a local deformation algebra with residue field $\DK$. As before, we consider $\tau$ as an element in $\hfhd_A=\Hom_A(\hfh\otimes_\DC A,A)=\hfhd\otimes_\DC A$. Let $(\cdot,\cdot)_A\colon \hfhd_A\times\hfhd_A\to A$ be the $A$-bilinear extension of the bilinear form $(\cdot,\cdot)\colon\hfhd\times\hfhd\to\DC$ that is induced by the usual non-degenerate, invariant bilinear form on $\hfg$ (cf.~\cite{Kac}).
Now 
let us consider the partial order $\uparrow_A$ on $\hfhd$ that is generated by $\mu\uparrow_A\lambda$ if there exists a positive root $\beta\in\hR^+$ and some $n\in\DN$ such that $2(\lambda+\tau+\rho,\beta)_\DK=n(\beta,\beta)_\DK$ and $\mu=\lambda-n\beta$. 
\begin{theorem}[\cite{KK79}] We have $[\Delta_\DK(\lambda):L_\DK(\mu)]\ne 0$ if and only if $\mu\uparrow_A \lambda$. 
\end{theorem}
In particular, the equivalence relation $\sim_A$ is generated by the partial order $\uparrow_A$. The following lemma is immediate from the above theorem.

\begin{lemma}
If $\lambda\sim_A\mu$, then $
\{\alpha\in \hR\mid 2(\lambda+\tau+\rho,\alpha)_\DK\in\DZ(\alpha,\alpha)_\DK\}=\{\alpha\in \hR\mid 2(\mu+\tau+\rho,\alpha)_\DK\in\DZ(\alpha,\alpha)_\DK\}$.
\end{lemma}
Hence any equivalence class $\Lambda\in\hfhd$ defines
\begin{align*}
\hR_A(\Lambda)&:=\{\alpha\in \hR\mid 2(\lambda+\tau+\rho,\alpha)_\DK\in\DZ(\alpha,\alpha)_\DK\text{ for some $\lambda\in\Lambda$}\},\\
&=\{\alpha\in \hR\mid 2(\lambda+\tau+\rho,\alpha)_\DK\in\DZ(\alpha,\alpha)_\DK\text{ for all $\lambda\in\Lambda$}\}.
\end{align*}
We also define
$$
\hCW_A(\Lambda):=\langle s_{\alpha+n\delta}\mid \alpha+n\delta\in \hR^{\re}\cap\hR_A(\Lambda)\rangle.
$$

Clearly, the elements in a fixed equivalence class $\Lambda$ have the same level, so we can talk about the level of an equivalence class. Note that an equivalence class if of critical level if and only if $\delta\in\hR_A(\Lambda)$, i.e.~ if $(\lambda+\rho,\delta)=(\delta,\delta)=0$ for all $\lambda\in\Lambda$. This is the case if and only if $n\delta\in\hR_A(\Lambda)$ for all $n\ne 0$. In this case, $\hCW_A(\Lambda)$ is an affine Weyl group isomorphic to the affinization of its finite analogue $\CW_A(\Lambda)$ that is generated by the reflections $s_\alpha$ for all finite roots $\alpha$ in $\hR_A(\Lambda)$.

The Kac--Kazhdan theorem now immediately implies the following. 
\begin{theorem}[The non-restricted linkage principle] Let $\Lambda\subset\hfhd$ be an equivalence class with respect to $\sim_A$. 
\begin{enumerate}
\item  If $\Lambda$ is non-critical, then $\Lambda$ is a $\hCW_A(\Lambda)$-orbit in $\hfhd$.
\item If $\Lambda$ is critical, then $\Lambda$ is an orbit under $\hCW_A(\Lambda)\times\DZ\delta$. 
\end{enumerate}
\end{theorem}

\subsection{Base change}
Now we explain one of the main reasons for the use of  the deformation theory in our approach.  Let us look at the special case $A=\tS=S(\fh)_{(0)}$, the localization of the symmetric algebra $S(\fh)$ at the maximal ideal generated by $\fh\subset S(\fh)$. For any prime ideal $\fp$ of $\tS$ we denote by $\tS_\fp$ the localization of $\tS$ at $\fp$ and by $\DK_\fp$ the residue field of $\tS_\fp$. Then $\tS$ is the intersection of $\tS_\fp$ inside the quotient field $\tQ$ of $\tS$ for all prime ideals of height one.

\begin{proposition} \label{prop-intequ} The equivalence relation $\ol\sim_{\tS}$ is the finest relation on $\hfhd$ that is coarser than $\ol\sim_{\tS_\fp}$ for all prime ideals $\fp$ of $\tS$ of height one.
 \end{proposition}
\begin{proof} Recall that the equivalence relation $\sim_A$ is generated by $\lambda\sim_A \mu$ if $(\rP_A^\CJ(\lambda):\rDelta_A(\mu))\ne 0$ for some open, locally bounded set $\CJ$. If $A\to A^\prime$ is a homomorphism of deformation algebras, then $\rP_A^\CJ(\lambda)\otimes_A A^\prime$ is projective in $\rCO_{A^\prime}^\CJ$, hence splits into a direct sum of restricted projective covers. Hence $\lambda\ol\sim_{\tS_\fp}\mu$ implies $\lambda\ol\sim_{\tS}\mu$. Let $\ol\sim^\prime_{\tS}$ be the finest relation on $\hfhd$ that is coarser than $\ol\sim_{\tS_\fp}$ for all prime ideals $\fp$ of $\tS$ of height one. Let $\Lambda$ be an equivalence class with respect to $\ol\sim_{\tS}$. Then $\Lambda$ is a union of equivalence classes with respect to $\ol\sim^\prime_{\tS}$. Let us write this decompositon as $\Lambda=\bigcup_{i\in I} \Lambda_i$. 

Now $\rP_\tS^\CJ(\mu)\otimes_\tS \tQ$ splits into a direct sum of Verma modules in $\rCO_\tQ$, hence we have a {\em canonical} decomposition
$$
\rP_\tS^\CJ(\mu)\otimes_\tS \tQ=\bigoplus_{i\in I} P_i,
$$
where $P_i$ is the direct summand that contains all Verma modules with highest weight belonging to $\Lambda_i$. By our assumption on $\ol\sim^\prime_{\tS}$, this direct sum decomposition induces a direct sum decomposition of $\rP_\tS^\CJ(\mu)\otimes_{\tS}\tS_\fp$  for each prime ideal $\fp$ of height one, i.e.~
$$
\rP_\tS^\CJ(\mu)\otimes_{\tS}\tS_\fp=\bigoplus_{i\in I}\left( \rP_\tS^\CJ(\mu)\otimes_{\tS}\tS_\fp\cap P_i\right).
$$
After taking the intersection we get a direct sum decomposition
$$
\rP_\tS^\CJ(\mu)=\bigoplus_{i\in I}\left( \rP_\tS^\CJ(\mu)\cap P_i\right)
$$
and we deduce that only one direct summand on the right hand side appears, i.e.~that $\Lambda$ is already an equivalence class with respec to $\ol\sim^\prime_{\tS}$. So we have  $\ol\sim_\tS=\ol\sim^\prime_\tS$.
\end{proof}

The advantage now is that the equivalence relations $\sim_{\tS_\fp}$ can be  described explicitely.

 \subsection{The restricted linkage principle}
 
 In the restricted case, we do not yet have such an explicit description of the highest weights of simple subquotients of a given Verma module as we have, by the Kac--Kazhdan theorem, in the non-restricted case. Nevertheless, we can explicitely determine the equivalence relations $\ol\sim_{\tS_\fp}$ for each prime ideal $\fp$ of height one and then use Proposition \ref{prop-intequ}.
 
Let $\lambda\in\hfhd$ be a weight at critical level and define
$R(\lambda)=\{\alpha\in R\mid \langle\lambda,\alpha^\vee\rangle\in\DZ\}$ (note that this definition only refers to finite roots!). For any $\alpha\in R(\lambda)$ we denote by $\alpha\downarrow\lambda$ the maximal element in $\{s_{\alpha+n\delta}.\lambda\mid n\in\DZ\}$ that is smaller or equal to $\lambda$.
Here is our first result:

\begin{theorem}[\cite{AF08}]\label{theorem-subgen} Let $\lambda\in\hfhd$ be of critical level and let $\fp\subset \tS$ be a prime ideal of height one. 
\begin{enumerate}
\item If $\alpha^\vee\not\in\fp$ for all $\alpha\in R(\lambda)$, then $\Delta_{\DK_\fp}(\lambda)$ is simple, i.e.~
$$
[\Delta_{\DK_\fp}(\lambda):L_{\DK_\fp}(\mu)]=
\begin{cases}
1,&\text{ if $\lambda=\mu$},\\
0,&\text{ otherwise.}
\end{cases}
$$
In particular, $\lambda\sim_{\tS_\fp}\mu$ implies $\lambda=\mu$.
\item If $\alpha^\vee\in\fp$ for some $\alpha\in R(\lambda)$, then $\fp=\alpha^\vee\tS$ and we have
 $$
[\Delta_{\DK_\fp}(\lambda):L_{\DK_\fp}(\mu)]=
\begin{cases}
1,&\text{ if $\mu\in\{\lambda,\alpha\downarrow\lambda\}$},\\
0,&\text{ otherwise.}
\end{cases}
$$
In particular, the equivalence class of $\lambda$ with respect to $\sim_{\tS_\fp}$ is the orbit of $\lambda$ under the subgroup of $\hCW$ that is generated by the reflections $s_{\alpha+n\delta}$  with $n\in\DZ$.  
\end{enumerate}
\end{theorem}

We can now deduce the restricted linkage principle. Let $\ol\Lambda\subset\hfhd$ be a critical, restricted equivalence class and define 
$$
\hCW(\ol\Lambda):=\langle s_{\alpha+n\delta}\mid \text{ $\alpha\in R(\lambda)$ for some $\lambda\in\ol\Lambda$ and $n\in\DZ$}\rangle.
$$

\begin{theorem}[\cite{AF09}] Suppose that $\ol\Lambda$ is a restricted critical equivalence class. Then $\ol\Lambda$ is a $\hCW(\ol\Lambda)$-orbit in $\hfhd$.
\end{theorem}
\begin{proof} Let $\lambda\in\ol\Lambda$. Let $\fp\subset \tS$ be a prime ideal of height one. If $\alpha^\vee\in\fp$ for some $\alpha\in R(\lambda)$, then $\fp=\alpha^\vee\tS$ and Theorem \ref{theorem-subgen} implies that the $\ol\sim_{\tS_\fp}$-equivalence class of $\lambda$ is its $\langle s_{\alpha+n\delta}\mid n\in\DZ\rangle$-orbit. If $\alpha^\vee\not\in\fp$ for all $\alpha\in R(\lambda)$, then $\lambda$ forms an $\ol\sim_{\tS_\fp}$-equivalence class by itself, again by Theorem \ref{theorem-subgen}. As by Proposition \ref{prop-intequ} the relation $\ol\sim_{\tS}$ is generated by the relations $\ol\sim_{\tS_\fp}$, the equivalence class of $\lambda$ is its $\hCW(\ol\Lambda)$-orbit. 
\end{proof}

\end{document}